\documentclass[11pt, reqno]{amsart}
\usepackage{amsmath,amssymb,amscd,mathrsfs,amscd}
\usepackage{graphics,verbatim, tikz}

\usetikzlibrary{matrix,arrows,patterns,decorations.pathreplacing}
\usepackage{amsmath,amstext,amsbsy,amssymb,amscd}
\usepackage{amsmath}
\usepackage{amsxtra}
\usepackage{amscd}
\usepackage{amsthm}
\usepackage{amsfonts}
\usepackage{graphicx}%
\usepackage{amssymb}
\usepackage{eucal}
\usepackage{color}
\usepackage{multirow}
\usepackage[enableskew,vcentermath]{youngtab}
\newtheorem{thm}[subsubsection]{Theorem}
\theoremstyle{plain}

\newtheorem{prop}[subsubsection]{Proposition}

\newtheorem{cor}[subsubsection]{Corollary}

\newtheorem{example}[subsubsection]{Example}

\newtheorem{rmk}[subsubsection]{Remark}

\theoremstyle{remark}

\numberwithin{equation}{section}

\newcommand{\C}{\mathbb C}

\newcommand{\gl}{\mathfrak{gl}}

\newcommand{\h}{\mathfrak{h}}

\newcommand{\del}{\delta}

\newcommand{\ep}{\epsilon}

\newcommand{\la}{\lambda}

\newcommand{\p}{\mathfrak p}
\newcommand{\borel}{\mathfrak b}
\newcommand{\stborel}{\mathfrak b^{\rm st}}

\newcommand{\barone}{\bar{1}}
\newcommand{\bartwo}{\bar{2}}

\newcommand{\sund}{\underline{s}}

\newcommand{\set}[1]{\left\{#1\right\}}

\renewcommand{\bar}[1]{\overline{#1}}

\begin{document}
\title[Super tableaux and a branching rule]{\sc Super tableaux and a branching rule 
for the \\general linear Lie superalgebra}

\author[Clark, Peng, Thamrongpairoj]{Sean Clark, Yung-Ning Peng, and S. Kuang Thamrongpairoj}
\address{Department of Mathematics, University of Virginia, Charlottesville, VA 22904, USA}
\email{sic5ag@virginia.edu (Clark), \quad st5tp@virginia.edu (Thamrongpairoj)}
\address{Institute of Mathematics, Academia Sinica, Taipei City 10617, Taiwan}
\email{ynp@math.sinica.edu.tw (Peng)}

\subjclass[2010]{}
 
\begin{abstract}
In this note, we formulate and prove a branching rule for simple 
polynomial modules of the Lie superalgebra $\mathfrak{gl}(m|n)$. Our
branching rules depend on the conjugacy class of the Borel subalgebra. 
A Gelfand-Tsetlin basis of a polynomial module associated
to each Borel subalgebra is obtained in terms
of generalized semistandard tableaux.
\end{abstract}


\maketitle


\section{Introduction}

Let $\gl(m)$ be the general linear Lie algebra and consider
the subalgebras
\[\gl(1)\subset \gl(2)\subset\ldots \subset \gl(m),\]
where we embed $\gl(k)\subset \gl(k+1)$ as the
collection of matrices whose last row and column are zero.
Then given a finite-dimensional simple $\gl(m)$-module 
$L_m(\la)$, we may
restrict the action to view $L_m(\la)$ as a $\gl(k)$-module and decompose it into
simple $\gl(k)$-modules $L_k(\mu)$; when $k=m-1$, this decomposition
is called a branching rule.
When $k=1$, this decomposition describes a linear basis 
for $L_{m}(\la)$, which is 
called the Gelfand-Tsetlin basis (\cite[Chap. 8]{GW}). 
There are several generalizations and
applications of the Gelfand-Tsetlin basis (\cite{Mo}). 
Moreover, this basis has an innate combinatorial description in terms of
semistandard Young tableaux, and hence can be used to derive some
combinatorial identities.

In this note, we generalize these concepts to the general linear
Lie superalgebra $\gl(m|n)$. It is known that there is a branching rule
for the standard Borel subalgebra of upper-triangular matrices
(\cite{BR}). Using Howe
duality (\cite{CW2,Se1}), we formulate and prove a branching law with respect to an
arbitrary choice of Borel subalgebra $\borel$, and in particular our approach
provides a novel proof for the standard Borel case.
We then deduce the existence of a Gelfand-Tsetlin basis parametrized
by certain tableaux, which we call ``$\borel$-semistandard''.
This provides a representation-theoretic proof that the number of these tableaux
is independent of the choice of $\borel$ (\cite{Kw}).

\vspace{.2cm}  {\bf Acknowledgement.} We are grateful to Weiqiang Wang for his
suggestion for this collaboration, and for the guidance and advice he has given
over the course of this project.  We would also like to thank Shun-Jen Cheng for 
his many helpful conversations. The third author was supported by an REU under 
Wang's NSF grant.

\section{Preliminaries} 

In this note, the underlying field is always $\C$, the complex numbers.

\subsection{Partitions, Young Diagrams, and Tableaux} 
Let $\lambda$ be a partition of $n$; that is, a sequence
$\lambda=(\lambda_1,\lambda_2,\ldots)$ of non-negative integers 
with $\lambda_i\geq \lambda_{i+1}$ and $|\lambda|=\sum_i \lambda_i=n$. 
The {\em length of $\la$}, denoted by $l(\la)$, 
is defined to be the number of non-zero entries of $\la$. 
Each partition corresponds to a unique Young diagram,
and we will freely identify a partition with 
its corresponding Young diagram.
We denote the conjugate partition to $\lambda$ by $\lambda'$.

A partition $\la$ is called an {\em $(m|n)$-hook partition}
if $\la_{m+1} \le n$. 
We denote the collection of $(m|n)$-hook partitions by $P_{m|n}$.

Given partitions $\lambda$ and $\mu$ such that $\lambda_i\geq \mu_i$,
let $\lambda/\mu$ denote the skew Young diagram. 
We will say that a skew diagram is a {\em horizontal strip} (respectively, {\em vertical strip})
if each column (respectively, row) of the diagram contains exactly one box.
Recall that given two partitions $\lambda$ and $\mu$, we say 
$\mu$ interlaces $\lambda$ if $\lambda_i\geq \mu_i\geq \lambda_{i-1}$
for all $i$. Then a skew diagram $\lambda/\mu$ is a horizontal strip (respectively,
a vertical strip) if and only if $\mu$ interlaces $\la$ (respectively,
$\mu'$ interlaces $\la'$). 

Let $A$ be a set. 
A {\em Young tableau with entries in $A$} is a Young diagram with an
element of $A$ inserted in each box; we will call $A$ an {\em alphabet},
and an element of $A$ a {\em letter}. 
Let $\lambda$ be a partition and $\mu$ be an arbitrary 
sequence of non-negative integers indexed by $A$ such that 
$\sum_{a\in A} \mu_a=|\lambda|$.
By a  {\em Young tableau of shape $\la$ and content $\mu$}, we mean a 
Young tableau corresponding to partition $\la$ such that each letter $a\in A$
appears in  $\mu_a$ boxes. When convenient, we will also use the notation
$\mu=(a^{\mu_a})_{a\in A}$.

\subsection{Littlewood-Richardson coefficients}
Let $\sigma$ and $\mu$ be partitions of length at most $m$.
Given a third partition $\lambda$, let $c_{\sigma\mu}^{\lambda}$
be the Littlewood-Richardson coefficients. One way to interpret these 
non-negative integers is as multiplicities:
\begin{equation}\label{LR1}
L_m(\sigma)\otimes L_m(\mu)\cong\bigoplus _{\lambda}L_m(\lambda)^{\bigoplus c^{\lambda}_{\sigma\mu}}.
\end{equation}

\begin{prop}\label{LR2} Let $\lambda$, $\mu$ and $\sigma$ be partitions. Then
\begin{enumerate}
\item[(1)] $c_{\sigma\mu}^{\lambda}=c_{\mu\sigma}^{\lambda}$ 
and $c^{\lambda}_{\sigma\mu}$ is nonzero only if $|\mu|+|\sigma|=|\lambda|$.
\item[(2)] $c^{\lambda^\prime}_{\sigma^\prime\mu^\prime} = c^\lambda_{\sigma\mu}$.
\item[(3)] Suppose further that $\mu=(n)$, a one row partition. Then $c^{\lambda}_{\sigma (n)}$ is either 1 or zero. Moreover, $c^{\lambda}_{\sigma (n)}=1$ if and only if the skew shape $\lambda/\sigma$ is a horizontal strip.
\item[(4)] Suppose further that $\mu=(1^n)$, a one column partition. Then $c^{\lambda}_{\sigma (1^n)}$ is either 1 or zero. Moreover, $c^{\lambda}_{\sigma (1^n)}=1$ if and only if the skew shape $\lambda^\prime/\sigma^\prime$ is a horizontal strip.
\end{enumerate}
\end{prop}
We refer the reader to \cite[Appendix A]{CW1} or \cite{Sa} for the details.

\subsection{The Lie superalgebra $\gl(m|n)$}

The general linear Lie superalgebra
$\mathfrak{g}=\mathfrak{gl}(m|n)$ is the vector 
superspace of $(m+n)$ by $(m+n)$ complex matrices,
with the block-diagonal even subspace $\mathfrak{g}_{\bar{0}}=\gl(m)\oplus\gl(n)$.

The standard Cartan subalgebra $\mathfrak{h}$ 
is the set of diagonal matrices 
in $\mathfrak{gl}(m|n)$. Let $h_i:=E_{ii}\in\mathfrak{h}$ be
the elementary matrices with a $1$ in the $(i,i)$ entry and $0$ in all
other entries, for $1\leq i\leq m+n$. 
For $1 \le i \le m$, $1 \le j \le n$, let $\del_i, \ep_j$ be the elements in $\mathfrak{h}^*$ defined by

$$\del_i(h_k) = \Delta_{i,k}$$
$$\ep_j(h_l) = \Delta_{j+m,l}$$
where $\Delta_{x,y} = 1$ if $x=y,$ and $0$ otherwise.
The set $\set{\del_i, \ep_j | 1\leq i\leq m, 1\leq j\leq n}$ forms a basis of $\mathfrak{h}^*$. 

The root system $\Phi$ of $\mathfrak{g}$ is the set 
\[
\Phi=\set{\del_k - \del_l, \ep_i - \ep_j, \del_k - \ep_i| 1\leq k,\l \leq m, 
1\leq i,j\leq n }.
\]
The Weyl group of $\mathfrak{gl}(m|n)$ is isomorphic to $S_m\times S_n$, 
where $S_m$ permutes the $\del_i$'s and $S_n$ permutes the $\ep_j$'s.
Note that not all simple systems are $W$-conjugate.
In particular, there exist different choices
of positive systems (and hence different choices of Borel subalgebras)
that are not $W$-conjugate; see \cite[Section~1.3]{CW1}.

\subsection{Borel Subalgebras}
By an {\em $\epsilon$-$\delta$ sequence} for $\gl(m|n)$, we mean 
a sequence of $m$ (indistinguishable) $\delta$'s and $n$ (indistinguishable)
$\epsilon$'s. For a given $\ep$-$\delta$ sequence $\sund$, 
we can assign the indices $1,\ldots, m$
to the $\delta$ and the indices $1,\ldots, n$ to the $\epsilon$ 
by the increasing order in which they appear. Then $\Pi(\sund)$, the
set of sequential differences of the indexed letters, forms a set
of simple roots for $\Phi$. 
For example, $\Pi(\epsilon\delta\delta\epsilon)
=\set{\epsilon_1-\delta_1,\delta_1-\delta_2,\delta_2-\epsilon_2}$. 

It is easy to check that
any simple system is $W$-conjugate to $\Pi(\sund)$ for some $\sund$.
Moreover, if $\sund$ and $\sund'$ are distinct $\ep$-$\delta$ sequences, then $\Pi(\sund)$ 
and $\Pi(\sund')$ are not conjugate, so these sequences index
the simple systems of $\Phi$ up to conjugacy.

Let $\Phi^+(\sund)$ be the set of positive roots relative to $\Pi(\sund)$ for
some sequence $\sund$. Then we define the Borel subalgebra corresponding to $\sund$
to be \[\borel(\sund)=\h\oplus\bigoplus_{\alpha\in \Phi^+(\sund)} \gl(m|n)_\alpha.\]
Note that the even subalgebra $\borel(\sund)_{\bar{0}}$ is independent of $\sund$;
indeed, $\borel(\sund)_{\bar{0}}$ is exactly
the standard Borel subalgebra of $\gl(m|n)_{\bar{0}}=\gl(m)\oplus\gl(n)$.

Any Borel subalgebra is conjugate by inner automorphisms to $\borel(\sund)$ for
some $\sund$, and $\borel(\sund)$ is not conjugate to $\borel(\sund')$ 
if $\sund\neq \sund'$; therefore, for the remainder of this note a Borel
subalgebra will mean $\borel(\sund)$ for some $\ep$-$\delta$ sequence $\sund$,
and we will abuse notation and identify the subalgebra with its corresponding sequence.
For example, the {\em standard Borel subalgebra} of upper triangular matrices in 
$\mathfrak{gl}(m|n)$ is represented by the sequence
\[\stborel=\stackrel{m}{\overbrace{\delta\ldots\delta}}
\,\stackrel{n}{\overbrace{\epsilon\ldots\epsilon}};.\]

\subsection{Polynomial weights and hook partitions}

Let
$\omega=\sum_{i=1}^m\mu_i\delta_i+\sum_{j=1}^n\nu_j\ep_j\in\mathfrak{h}^*$.
We call $\omega$ a polynomial weight if $\mu_i$ and $\nu_j$ are non-negative
integers for all $i,j$.

Fix a Borel subalgebra $\borel$. Then
each $(m|n)$-hook partition $\lambda$ corresponds to a polynomial weight as follows.
For $1 \le i \le m$ and $1 \le j \le n$, let $d_i$ (resp. $e_j$) be the number
of $\epsilon$'s (resp. $\delta$'s) appearing before the $i^{\rm th}$ $\delta$ (resp.
$j^{\rm th}$ $\epsilon$) in $\borel$.
For example, if $\borel=\del \del \ep \del \ep \del$, then
$$d_1 = d_2 = 0, d_3 = 1, d_4 = 2, e_1 =2, e_2 = 3.$$

For $1\leq i \leq m$ and $1\leq j\leq n$, let
$p_i = \max \{ \lambda_i - d_i, 0\}$ and 
$q_j = \max \{ \lambda '_j - e_j, 0\}$.
The pair $(p_i,q_j)$  is called
the $\mathfrak{b}$-Frobenius coordinates; cf. \cite{CW1}. 
Then set
\[
\lambda^{\mathfrak{b}} = \sum_{i=1}^m p_i \del_i + \sum_{j=1}^n q_j \ep_j.
\] 

When $\borel=\stborel$, there is a simpler description.
Define the partitions $\mu=(\lambda_1,\ldots, \lambda_m)$ and 
$\sigma=(\lambda_{m+1}, \lambda_{m+2},\ldots)'$. Then we define $\lambda^{\natural}=\lambda^{\stborel}=\sum_{i=1}^m \mu_i\delta_i+\sum_{j=1}^{n}\sigma_j\epsilon_j$.

\subsection{Polynomial modules of $\mathfrak{gl}(m|n)$}

An $\h$-semisimple module $M$ is a 
polynomial module if the weights of $M$ are polynomial.
In this note, by polynomial module we always mean a finite-dimensional
polynomial module.

Given a weight $\omega$ and a Borel $\borel$, let
$L(\borel, \omega)$ be the simple module of highest weight $\omega$
with respect to $\borel$. When $\borel=\stborel$, we use the shorthand
$L(\omega)=L(\stborel, \omega)$. Then we have the following result
(cf. \cite{K,CK,CLW,CW1}).

\begin{prop}\label{prop:classification}
Fix a Borel subalgebra $\borel$.
\begin{enumerate}
\item[(1)] $\set{L_{m|n}(\borel,\lambda^\borel):\lambda\in P_{m|n}}$ is
a complete list of pairwise non-isomorphic (finite-dimensional) simple
polynomial $\gl(m|n)$-modules.
\item[(2)]  $L_{m|n}(\borel, \lambda^{\borel})\cong L_{m|n}(\lambda^\natural)$.
\item[(3)]  The category of polynomial modules of $\mathfrak{gl}(m|n)$ is a semisimple tensor category.
\end{enumerate}
\end{prop}

\section{Branching rule for $\mathfrak{gl}(m|n)$}

\subsection{Branching Rule}
Fix a Borel subalgebra $\borel$.
Let $\hat{\mathfrak{b}}$ be the $\ep$-$\delta$ sequence 
obtained from $\mathfrak{b}$ by deleting the last entry 
in the sequence $\borel$, and $\hat{m}$ (respectively, $\hat{n}$) 
be the number of $\delta$'s (respectively, $\epsilon$'s) in $\hat{\borel}$.
Also set $\hat{h}=h_m\in \mathfrak h$ if $\hat{n}=n$ and $\hat{h}=h_{m+n}$
if $\hat{m}=m$.
Then we have a natural embedding 
\[\gl(\hat{m}|\hat{n})\cong \hat{\mathfrak h}\oplus \bigoplus_{\alpha\in \Phi\cap \hat{\mathfrak h}^*} \gl(m|n)_\alpha\] where $\hat{\mathfrak{h}}\oplus \C \hat{h}=\mathfrak h$.
With respect to this embedding, we have the following branching 
rule for an arbitrary Borel subalgebra.

\begin{thm}[Branching Rule for $\mathfrak{gl}(m|n)$ with respect to a general Borel] \label{thm:branching}
Let $\lambda$ be an $(m|n)$-hook partition and $\borel$ be a choice
of Borel subalgebra for $\gl(m|n)$.
In the above notation,
\[L_{m|n}(\borel,\lambda^{\borel})\downarrow_{\gl(\hat m, \hat n)}\cong\bigoplus_{\sigma}
L_{\hat{m}|\hat{n}}(\hat{\mathfrak{b}}, \sigma^{\hat{\mathfrak{b}}})\]
where the sum is taken over all $(\hat m|\hat n)$-hook partitions $\sigma$
such that 
\begin{enumerate}
\item[(1)] $\sigma'$ interlaces $\lambda'$ if $m=\hat m$,
\item[(2)] $\sigma$ interlaces $\lambda$ if $n=\hat n$.
\end{enumerate}
\end{thm}
\begin{proof}

First recall the statement of super Howe duality:
as $\mathfrak{gl}(m|n) \times \mathfrak{gl}(k)$-modules,
\begin{equation} \label{first decompose}
S(\C^{m|n} \otimes \C^k) \cong \bigoplus_{\substack{l( \la ) \le k \\ \la_{m+1} \le n}} L_{m|n}(\borel, \la^\borel ) \otimes  L_k(\la).
\end{equation}

Let $\p=\mathfrak{gl}(\hat m|\hat n) \times \mathfrak{gl}(m-\hat{m}|n-\hat{n})$.
Restricting \eqref{first decompose} to a
$\p\times \gl(k)$-module yields the decomposition
\begin{equation} \label{first branching}
S(\C^{m|n} \otimes \C^k) \downarrow_{\p\times \gl(k)} \cong \bigoplus_{\substack{l( \la ) \le k \\ \la_{m+1} \le n}} L_{m|n}(\borel, \la^{\borel} ) \otimes  L_k(\la).
\end{equation}

On the other hand, evidently $S(\C^{m|n} \otimes \C^k)\downarrow_{\p\times \gl(k)}\cong S(\C^{\hat{m}|\hat{n}} \otimes \C^k) \otimes S(\C^{m-\hat m|n-\hat n} \otimes \C^k)$.
By applying Super Howe duality again to each tensor factor, we see that

\begin{equation*} \label{second decompose}
S(\C^{m|n} \otimes \C^k)\downarrow_{\p\times \gl(k)} \cong \bigoplus_{\sigma, \mu} L_{\hat m|\hat n}(\hat\borel, \sigma^{\hat\borel}) \otimes L_k(\sigma) \otimes L_{m-\hat m|n-\hat n}(\theta, \mu^{\theta}) \otimes L_k(\mu),
\end{equation*}
where the sum is taken over all $\sigma$ and $\mu$ such that 
$\sigma$ is an $(\hat m|\hat n)$-hook partition with $l(\sigma)\leq k$ 
and $\mu$ is a $(m-\hat m|n-\hat n)$-hook partition with $l(\mu)\leq k$,
and the notation $\theta$=$\del$(resp. $\ep$) if $\hat n=n$(resp. $\hat m=m$).
In particular, since $\gl(m-\hat m|n-\hat n)\cong \C$, $L_{m-\hat m|n-\hat n}(\theta, \mu^{\theta})$
is $1$-dimensional and thus
\begin{equation} \label{second decompose2}
S(\C^{m|n} \otimes \C^k)\downarrow_{\gl(\hat m|\hat n)\times \gl(k)}
= \bigoplus_{\sigma, \mu} 
L_{\hat m|\hat n}(\hat\borel, \sigma^{\hat\borel})
\otimes \big(L_k(\sigma)  \otimes L_k(\mu)\big).
\end{equation}
Then using \eqref{LR1},
\begin{eqnarray} \label{second branching2}
\nonumber S(\C^{m|n} \otimes \C^k)\downarrow_{\gl(\hat m|\hat n)\times \gl(k)}  &= \bigoplus_{\sigma, \mu}L_{\hat m|\hat n}(\hat\borel, \sigma^{\hat\borel}) \otimes \bigoplus_\lambda L_k(\lambda)^{\oplus c^\lambda_{\sigma\mu}}\\
 & = \bigoplus_{\sigma,\mu,\lambda}L_{\hat m|\hat n}(\hat\borel, \sigma^{\hat\borel}) \otimes L_k(\lambda)^{\oplus c^\lambda_{\sigma\mu}}
\end{eqnarray}
Comparing \eqref{first branching} and \eqref{second branching2}, we obtain
\begin{equation}\label{branchingequality}
L_{m|n}(\borel, \la^\borel )\downarrow_{\mathfrak{gl}(\hat m|\hat n)} = 
\bigoplus_{\sigma,\mu}  L_{\hat m|\hat n}(\hat\borel, \sigma^{\hat\borel})^{\oplus c^\lambda_{\sigma\mu} }.
\end{equation}

Since $\{m-\hat{m},n-\hat n\}=\set{0,1}$, $\mu$ is either a row or column partition.
The result follows by Proposition \ref{LR2}.

\end{proof}

\begin{rmk}
The case $\borel=\stborel$ had been proved in \cite{BR}
using the commuting actions of $\gl(m|n)$ and $S_d$ 
on $(\C^{(m|n)})^{\otimes d}$.
\end{rmk}

\subsection{The Gelfand-Tsetlin basis and $\borel$-semistandard tableaux}

By iterating Theorem \ref{thm:branching}, one obtains a linear
basis of $L_{m|n}(\borel, \lambda^\borel)$. These basis vectors are in
bijection with the sequences of triples $(m_k,n_k,\lambda^k)$
where $m_k,n_k$ are non-negative integers and $\lambda^k$ are partitions
satisfying the following conditions.
\begin{enumerate}
\item Either $m_k=m_{k+1}$ and $n_k=n_{k+1}+1$ 
		or $m_k=m_{k+1}+1$ and $n_k=n_{k+1}$.
\item $\lambda^k$ is an $(m_k|n_k)$-hook partition.
\item $\lambda^{k}/\lambda^{k+1}$ is a vertical strip if $m_k=m_{k+1}$
or a horizontal strip otherwise (or the empty diagram).
\end{enumerate}

These sequences may be realized in terms of certain tableaux of $\lambda$.
Let $A$ denote the alphabet set $\set{1,\ldots, m, \bar{1}, \ldots, \bar{n}}$
and $\mathfrak{b}$ be the following $\ep$-$\delta$ sequence:
\[
\mathfrak{b}=\overbrace{\delta\delta\ldots\delta}^{\mu_1}
\overbrace{\ep\ep\ldots\ep}^{\nu_1}\ldots\ldots
\overbrace{\delta\ldots\delta}^{\mu_t}\overbrace{\ep\ldots\ep}^{\nu_t} \,,
\]
where $\mu_1$ and $\nu_t$ are non-negative and all other $\mu_i$, $\nu_j$ are positive integers. For convenience, set $\nu_0=0$.  Define the total order $<_\borel$ on $A$ by setting 
\[i<_\borel j\text{ if }1\leq i<j\leq m,\qquad \bar k<_\borel \bar l\text{ if }1\leq k<l\leq n,\]
\[i<_\borel \bar{k} \text{ if }i\leq\sum_{g=1}^r\mu_g\text{ and } k>\sum_{g=0}^{r-1}\nu_{g}\text{ for some }1\leq r\leq t,\]
\[\bar{k}<_\borel i \text{ if }i>\sum_{g=1}^{r}\mu_g\text{ and } k\leq \sum_{g=0}^{r}\nu_{g}\text{ for some }1\leq r\leq t.\]
A tableau in $A$ will be called 
{\em $\mathfrak{b}$-semistandard} if the following three conditions are satisfied:

\begin{enumerate}
\item The entries are weakly increasing along each row and column
with respect to $<_\borel$,

\item The entries from $\set{1,2, \ldots , m}$ are strictly increasing along each column,
\item The entries from $\set{\bar{1},\bar{2},\ldots,\bar{n}}$ are strictly increasing along each row.
\end{enumerate}
The combinatorics of such tableaux have been studied in some detail,
and in particular it is known that 
there is a content-preserving bijection between $\borel$-semistandard
tableaux and $\stborel$-semistandard tableaux; cf. \cite{Kw}.

\begin{prop}
The module $L_{m|n}(\borel, \lambda^{\borel})$ has a basis
indexed by $\borel$-semistandard tableaux of shape $\lambda$
which is obtained by iterated application of the branching rule.
The entries determine which simple summand the basis vector
belongs to at each application of the branching rule.
\end{prop}

\begin{proof}
Order $A$ from largest to smallest according to $\borel$ to obtain
$A=\set{a_1>_\borel \ldots >_\borel a_{n+m}}$.
Then given a sequence $(m_k,n_k,\lambda^k)$ associated to a linear basis element
we can define a unique $\borel$-semistandard tableau: give each box in 
$\lambda^{k}/\lambda^{k+1}$ the entry $a_k$. Clearly the resulting
tableau is uniquely determined by the sequence of partitions. On the other hand,
given a $\borel$-semistandard tableau it is evident that
stripping away the boxes with entry $a_1$, then those with $a_2$,
etc. defines a sequence of partitions corresponding to a basis element.
\end{proof}

\begin{example}
Consider $\mathfrak{gl} (3|2)$ and $\mathfrak{b}=\delta\ep\ep\delta\delta$. We choose the following partitions
$\lambda = \lambda^{1} = (4,3,3,2,2,1)$, 
$\lambda^{2} = (3,3,2,2,2)$,
$\lambda^{3} = (3,2,2,2)$, 
$\lambda^{4} = (2,1,1,1)$,
$\lambda^{5} = (1)$. 

In this case, the ordered alphabet is
\[1<_{\borel}\bar{1}<_{\borel}\bar{2}<_{\borel}2<_{\borel}3.\]
Following the algorithm above, we find the corresponding 
$\mathfrak{b}$-semistandard tableau to be
$$\young(1\barone\bartwo3,\barone\bartwo2,\barone\bartwo3,\barone\bartwo,22,3)$$
\end{example}

We note that if $\mu$ is the content of a $\borel$-semistandard tableau $T$ and we write
\[\mu=(1^{\mu_1},2^{\mu_2},\ldots, m^{\mu_m},\bar{1}^{\mu_{\bar{1}}},\ldots, \bar{n}^{\mu_{\bar{n}}})\] 
then the basis vector corresponding to $T$ 
has weight 
\[\mu=\sum_{i=1}^m \mu_i\delta_i + \sum_{j=1}^n \mu_{\bar{j}}\epsilon_j.\]

\subsection{An analogue of the Kostka numbers}

The branching rule implies that the number of 
$\mathfrak{b}$-semistandard tableaux of shape $\la$ is equal to the dimension 
of $\mathfrak{gl}(m|n)$-module $L_{m|n}(\mathfrak{b},\la^\mathfrak{b})$, 
which is the same module as $L_{m|n}(\lambda^{\natural})$
by Proposition \ref{prop:classification}. Comparing the Gelfand-Tsetlin bases,
we obtain the following result, which is a special case of \cite[Proposition 2.11]{Kw}.

\begin{cor}\label{lastcor}
Fix an $(m|n)$-hook partition $\la$ and a content $\mu$. 
The number $K_{\lambda,\mu}^\borel$ of 
$\borel$-semistandard tableaux of shape $\lambda$
and content $\mu$ is independent of the choice of $\mathfrak{b}$.
In particular, there is a well-defined number $K_{\lambda,\mu}^{(m|n)}=K_{\lambda,\mu}^{\borel}$ which is an analogue of the Kostka number $K_{\lambda,\mu}$.
\end{cor}

\begin{example}
Consider $\mathfrak{gl}(2|1)$ and $\lambda = (3,2,1)$

$$\young(\,\,\,,\,\,,\,)$$

Consider two different $\ep$-$\del$ sequences $\mathfrak{b}_1 = \mathfrak{b}^{st} = \del \del \ep$ and $\mathfrak{b}_2 = \del \ep \del$. One may check that all possible $\mathfrak{b}_1$-semistandard tableaux of shape $\lambda$ are
$$\young(111,22,\barone) \hspace{1 cm} \young(111,2\barone,\barone) \hspace{1 cm} \young(112,22,\barone) \hspace{1 cm} \young(11\barone,2\barone,\barone)$$
$$\young(112,2\barone,\barone) \hspace{1 cm} \young(11\barone,22,\barone) \hspace{1 cm} \young(122,2\barone,\barone) \hspace{1 cm} \young(12\barone,2\barone,\barone)$$
and all possible $\mathfrak{b}_2$-semistandard tableaux are
$$\young(111,\barone 2,2) \hspace{1 cm} \young(111,\barone 2,\barone) \hspace{1 cm} \young(112,\barone 2,2) \hspace{1 cm} \young(11\barone,\barone 2,\barone)$$
$$\young(11\barone,\barone 2,2) \hspace{1 cm} \young(11\barone,\barone 2,2) \hspace{1 cm} \young(1\barone 2,\barone 2,2) \hspace{1 cm} \young(1\barone 2,\barone 2,\barone)$$
Hence the number of semistandard tableaux are the same.

In both cases, there are two semistandard tableaux with content $(1^2,2^2 |\, \bar{1}^2)$ and there is only one semistandard tableau for other possible contents.

\end{example}

\end{document}